\documentclass[11pt]{amsart}
\usepackage{mathrsfs,amssymb}
\usepackage{graphicx}
\usepackage{verbatim}
\begin{document}

\newtheorem{theorem}{Theorem}[section]
\newtheorem{proposition}[theorem]{Proposition}
\newtheorem{lemma}[theorem]{Lemma}
\newtheorem{corollary}[theorem]{Corollary}
\newtheorem{conjecture}[theorem]{Conjecture}
\newtheorem{question}[theorem]{Question}
\newtheorem{problem}[theorem]{Problem}
\theoremstyle{definition}
\newtheorem{definition}{Definition}
\newcommand{\tmop}[1]{\ensuremath{\operatorname{#1}}}

\theoremstyle{remark}
\newtheorem{remark}[theorem]{Remark}

\renewcommand{\labelenumi}{(\roman{enumi})}
\def\theenumi{\roman{enumi}}

\numberwithin{equation}{section}

\renewcommand{\Re}{\operatorname{Re}}
\renewcommand{\Im}{\operatorname{Im}}

\def\scrA{{\mathcal A}}
\def\scrB{{\mathcal B}}
\def\scrD{{\mathcal D}}
\def\scrL{{\mathcal L}}
\def\scrS{{\mathcal S}}

\def \G {{\Gamma}}
\def \g {{\gamma}}
\def \R {{\mathbb R}}
\def \H {{\mathbb H}}
\def \C {{\mathbb C}}
\def \Z {{\mathbb Z}}
\def \Q {{\mathbb Q}}
\def \TT {{\mathbb T}}
\newcommand{\T}{\mathbb T}
\def \GinfmodG {{\Gamma_{\!\!\infty}\!\!\setminus\!\Gamma}}
\def \GmodH {{\Gamma\setminus\H}}
\def \sl  {\hbox{SL}_2(\mathbb Z)}
\def \slr  {\hbox{SL}_2(\mathbb R)}
\def \psl  {\hbox{PSL}_2(\mathbb R)}

\newcommand{\mattwo}[4]
{\left(\begin{array}{cc}
                        #1  & #2   \\
                        #3 &  #4
                          \end{array}\right) }

\newcommand{\rum}[1] {\textup{L}^2\left( #1\right)}
\newcommand{\norm}[1]{\left\lVert #1 \right\rVert}
\newcommand{\abs}[1]{\left\lvert #1 \right\rvert}
\newcommand{\inprod}[2]{\left \langle #1,#2 \right\rangle}
\newcommand{\tr}[1] {\hbox{tr}\left( #1\right)}

\renewcommand{\^}[1]{\widehat{#1}}

\renewcommand{\i}{{\mathrm{i}}}

\date{\today}

\newcommand{\area}{\operatorname{area}}
\newcommand{\ecc}{\operatorname{ecc}}

\newcommand{\dom}{\operatorname{Dom}}
\newcommand{\Dom}{\operatorname{Dom}}

\newcommand{\Norm}{\mathcal N}
\newcommand{\simgeq}{\gtrsim}%
\newcommand{\simleq}{\lesssim}

\newcommand{\UN}{U_N}
\newcommand{\OPN}{\operatorname{Op}_N}
\newcommand{\HN}{\mathcal H_N}
\newcommand{\TN}{T_N}  
\newcommand{\PDO}{\Psi\mbox{DO}}
\newcommand{\Vol}{\operatorname{Vol}}
\newcommand{\vol}{\operatorname{vol}}
\newcommand{\dvol}{\operatorname{dvol}}

\newcommand{\dcap}{\operatorname{cap}}
\newcommand{\vE}{\mathcal E}
\newcommand{\ncap}{{\mathfrak n}}

\title{Small scale equidistribution of eigenfunctions  on the torus}
 
\author{Stephen Lester and Ze\'ev Rudnick  }
\address{Department of mathematics, KTH, SE-100 44  Stockholm, Sweden}
\email{sjlester@gmail.com}
\address{Raymond and Beverly Sackler School of Mathematical Sciences,
Tel Aviv University, Tel Aviv 69978, Israel}
\email{rudnick@post.tau.ac.il}

\begin{abstract}
We study the small scale distribution of the $L^2$ mass of eigenfunctions of the Laplacian on the flat torus $\TT^d$.
Given an orthonormal basis of eigenfunctions, we show
the existence of a density one subsequence whose $L^2$ mass equidistributes at small scales.
In dimension two our result holds all the way down to the Planck scale.  For   dimensions $d=3,4$  we can restrict to individual eigenspaces and show small scale equidistribution in that context.

We also study irregularities of quantum equidistribution: 
We construct  eigenfunctions whose $L^2$ mass does not equidistribute at all scales above the Planck scale.
Additionally, in  dimension $d=4$ we show the existence of eigenfunctions for which the proportion of $L^2$ mass
in small balls blows up at certain scales.

\end{abstract}
 
\date{\today}
\maketitle

\section{Introduction}

\subsection{The semiclassical eigenfunction hypothesis}
Let $M$ be a compact Riemannian manifold (smooth, connected and with no boundary), with associated Laplace-Beltrami operator $\Delta$, and $\{\psi_n\}$ an orthonormal basis of $L^2(M,\dvol)$ consisting of eigenfunctions: $-\Delta \psi_n = \lambda_n\psi_n$,  where $\dvol$ is the normalized Riemannian volume form.
If the geodesic flow is ergodic, 
the   Quantum Ergodicity Theorem \cite{Schnirelman, Zelditch, CdV} says that for any choice of orthonormal basis (ONB) 
$\{\psi_n\}$ consisting of eigenfunctions of the
Laplacian, there is a density one subsequence of these eigenfunctions which are uniformly distributed in the unit cotangent bundle $S^*M$, where a density one subsequence $\{ \psi_{n_\ell} \} \subset \{ \psi_n \}$ of eigenfunctions is one such that
\[
\lim_{\Lambda \rightarrow \infty} \frac{\#\{ \psi_{n_\ell} : \lambda_{n_\ell} \le \Lambda \}}{\# \{ \lambda_n \le \Lambda \}}=1.
\]
(For certain chaotic billiards,   exceptional eigenfunctions do exist, see \cite{Hassell}.) In particular, there is a density-one subsequence of the eigenfunctions so that the probability densities $|\psi_n(x)|^2$ converge weakly to the uniform distribution in configuration space $M$ along this subsequence, i.e. for any (nice)   fixed  subset $\Omega\subseteq M$ of positive measure,  
$$
\frac 1{\vol(\Omega)}\int_{\Omega} |\psi_n(x)|^2\dvol(x)\to 1 \;.
$$
Uniform distribution in configuration space is not only a feature of ergodicity: Marklof and Rudnick \cite{MR} show that this is also the case for  rational polygons, and for flat tori.

 M.V. Berry \cite{BerryJPhysA77, Berrysurvey} in his work on the ``Semiclassical Eigenfunction Hypothesis'' (see also \cite{Voros}), proposed to go beyond uniform distribution, and study the amplitudes  $|\psi_n(x)|^2$ when smoothed over  regions in $M$,   whose diameter shrinks as $\lambda_n\to \infty$, but at a rate slower than the Planck scale $\hbar\approx 1/\sqrt{\lambda_n}$,   
 that is  to study the local averages  
 \begin{equation}\label{local average}
 \frac 1{\vol B(x_n,r_n)} \int_{B(x_n, r_n)}|\psi_n(x)|^2 \dvol(x)
 \end{equation}
 where $B(x_n,r_n)$ is a geodesic ball of radius $r_n$ centered at $x_n\in M$,   
so that as $\lambda_n\to \infty$,  $r_n\to 0$, but  $r_n \sqrt{\lambda_n}\to \infty$. We will say that small scale equidistribution of the eigenfunctions $\{\psi_n\}$ holds if \eqref{local average} tends to $1$.

There are very few rigorous results  on small scale equidistribution in the literature.  
Luo and Sarnak \cite{LS95} studied the case of the modular surface, and the orthonormal set of eigenfunctions of the Laplacian which are eigenfunctions of all Hecke operators, showing that for these, small scale equidistribution holds
along a density one subsequence for radii $r\gg \lambda^{-\alpha}$, for some small $\alpha>0$. Under the assumption
of the Generalized Riemann Hypothesis
Young \cite{Young} showed that small scale
equidistribution holds for \textit{all} such eigenfunctions for radii $r \gg \lambda^{-1/4+o(1)}$.

The case of  compact  manifolds with negative sectional curvature was recently investigated independently by    Hezari and Rivi\`ere \cite{HRneg} and Han \cite{Han} who obtained commensurability of the masses along a density one subsequence for logarithmically small radii $r=(\log \lambda)^{-\alpha}$ $(0 \le \alpha  < \frac{1}{3 \tmop{dim} M}$): 

\begin{equation*}
a_1\leq \frac 1{ \vol(B(x_n, r_n)) } \int_{B(x_n,r_n)}|\psi_{n}(x)|^2 \dvol(x) \leq a_2
\end{equation*}
along the subsequence, where the constants $0<a_1<a_2$  are independent of the centers of the balls $x_n$ and of the subsequence.

 \subsection{Small scale equidistribution on the flat torus}

 The case of interest for us is  that   of the flat $d$-dimensional torus  $\TT^d=\R^d/2\pi \Z^d$. 
 The ``Semiclassical Eigenfunction Hypothesis'' predicts that \eqref{local average}  converges to $1$  
in this setting for radii $r_n \rightarrow 0$ with $r_n\sqrt{\lambda_n} \rightarrow \infty$, as $\lambda \rightarrow \infty$. 
 Hezari and Rivi\`ere \cite{HRtor} have recently studied small scale equidistribution in  $\TT^d$. They show that for a fixed center $x_0\in \TT^d$, for any ONB of eigenfunctions $\{\psi_n\}$, there is a density one subsequence 
so that for all balls $B(x_0,r_n)$ of radius $r_n>\lambda_n^{-\frac 1{4(d+1)}}$   
one has  that \eqref{local average} tends to $1$ along the subsequence. 

Note that below the Planck scale $r=\lambda^{-1/2}$, equidistribution
fails badly. 
For example, consider the ONB of eigenfunctions $\psi^-_{\mu}(x)=\sqrt{2}  \sin (\langle \mu, x \rangle)$, 
$\psi^+_{\mu}(x)=\sqrt{2}  \cos (\langle \mu, x \rangle)$, $\mu \in \mathbb Z^d/\{\pm 1\}$, with eigenvalue $\lambda=|\mu|^2$. For $r=o(\lambda^{-1/2})$, and $x \in B(0,r)$ one has $\psi^\pm_{\mu}(x)\sim \psi^\pm_{\mu}(0)=1\pm 1$, so that every eigenfunction in this ONB is not equidistributed below the Planck scale.

One of our goals is to prove small scale equidistribution  on $\TT^d$, uniformly for all not too small balls.   We succeed for radii $r_n\gg \lambda_n^{-\frac 1{2(d-1)}+o(1)}$, in particular in dimension $d=2$, our result extends all the way down to the Planck scale $r\gg \lambda^{-1/2+o(1)}$:
\begin{theorem}\label{thm:small balls}
Let $\{\psi_n\}$ be an orthonormal basis of eigenfunctions of $L^2(\TT^d,\dvol)$, and 
$$ 
\mathcal B_n=\left\{ B(y,r) \subset \TT^d : r >  \lambda_n^{\frac{-1}{2(d-1)}+o(1)} \right\}.
$$ 
Then along a density one subsequence, 
\[
\lim_{n\to \infty} \sup_{\substack{B(y,r) \in \mathcal B_n}} \left|
\frac {1}{\vol(B(y,r))} \int_{B(y,r)} |\psi_n(x)|^2 \, \dvol(x)- 1 \right|=0\;.
\]
\end{theorem}

This result gives that the $L^2$ mass of ``almost all'' eigenfunctions in the given orthonormal basis is uniformly distributed in every small ball $B(y,r)$. Even though our result
does not reach the Planck scale for dimensions $d>2$, the scale we achieve is actually optimal (up to the $\lambda^{o(1)}$ factor). This was pointed out to us by Jean Bourgain (see Remark~\ref{rem:Bourgain} after Theorem \ref{thm:3dloc}).

\subsection{Irregularities in quantum equidistribution}

Theorem~\ref{thm:small balls} leaves open the existence of exceptional sequences of eigenfunctions. In Theorem~\ref{prop:scar2} we show that these do exist, so that one cannot improve the ``almost all" statement. We show that  there is a sequence of eigenvalues $\lambda_n\to \infty$ and corresponding $L^2$-normalized eigenfunctions $\psi_n$ so that for any choice of radii $r_n$ so that $r_n\to 0$, but $r_n\sqrt{\lambda_n}\to \infty$, 
$$
\lim_{n\to \infty}
\frac 1{\vol(B(0,r_n))}\int_{B(0,r_n)}|\psi_n( x)|^2 \dvol(x) = 2  \;.
$$
For a fixed radius $r\approx 1$, see \cite{Jakobson} for information on possible ``quantum limits''. 

In   dimension  $d=4$, we can also create ``massive irregularities", where we find an infinite sequence of eigenvalues $\lambda_n\to \infty$,  so that given any sequence of  balls $B(x_n,r_n)$ of radius 
$r_n\ll \lambda^{-\frac 1{2(d-1)}-o(1)}$, there are normalized eigenfunctions $\psi_n$ whose $L^2$-mass on the specific balls $B(x_n,r_n)$ blows up: 
$$
\lim_{n\to \infty} \frac 1{\vol(B(x_n,r_n))} \int_{B(x_n,r_n)} |\psi_n(x)|^2 \, \dvol(x)=\infty \;.
$$
 A related feature was found on certain negatively curved surfaces by Iwaniec and Sarnak \cite{IS}, who found eigenfunctions of the Laplacian   whose values blow up at special points, see also \cite{Milicevic}.

On the other hand, in dimension $d=2$ we
rule out the existence of such ``massive irregularities"
at scales above $r>\lambda^{-1/4+o(1)}$ and
expect that they do not exist at all for $r>\lambda^{-1/2+o(1)}$, i.e. just above the Planck scale.
We will show that for every  eigenfunction $\psi(x)$ in dimension $d=2$ that 
for radii $r > \lambda^{-1/4+o(1)}$
\begin{equation} \label{eq:nodscar}
\sup_{y \in \TT^2}\frac{1}{\tmop{vol}(B(y,r))}
\int_{B(y,r)} |\psi(x)|^2 \dvol(x) \ll 1.
\end{equation}
The problem of obtaining an
upper bound for the proportion of $L^2$ mass of eigenfunctions
in small balls was previously studied by
Sogge \cite{Sogge},
who showed for any compact $d$-dimensional Riemannian manifold (smooth, connected and with no boundary) $M$ and an $L^2$-normalized eigenfunction of
the Laplace-Beltrami operator $\psi$ that  
\[
\sup_{y \in M}
\frac{1}{\tmop{vol}(B(y,r))}
\int_{B(y,r)} |\psi(x)|^2 \dvol(x) \ll r^{1-d},
\]
for $r> \lambda^{-1/2}$.

\subsection{Localizing on eigenspaces}\label{subsec:loc}

In higher dimensions ($d\geq 3$), the eigen-spaces have fairly large dimension, 
and we can also localize on each $\lambda$-eigenspace in dimensions $d= 3,4$.  
That is, prove analogues of Theorem~\ref{thm:small balls} when we restrict to an orthonormal basis of an individual eigenspace. 
For instance, in dimension $d=3$ for $\lambda \not\equiv 0,4,7 \pmod 8$, the dimension of the $\lambda$-eigenspace, which we denote $N_{\lambda}$, is quite large of size $\approx \lambda^{\tfrac12\pm o(1)}$; for $d=4$ and $\lambda$ odd, $\lambda\leq N_{\lambda}\ll \lambda^{1+o(1)}$.
Using results from the arithmetic theory of quadratic forms, we show 

\begin{theorem} \label{thm:3dloc}
Suppose that $d=3$ and $\lambda \not\equiv 0,4,7 \pmod 8$, or, $d=4$ and $\lambda$ is odd.
Let $\{\psi_n\}_{\lambda_n=\lambda}$
be an ONB of eigenfunctions of the $\lambda$-eigenspace
and let 
\[
\mathcal B_{\lambda}=\left\{ B(y,r) \subset \TT^d  
: r > \lambda^{-\frac{1}{2(d-1)}+o(1)} \right\}.
\]
Then there exists a subset   $S_\lambda \subseteq \{\psi_{n}\}_{\lambda_n=\lambda}$ of cardinality 
$N_{\lambda}(1+o(1))$, as $\lambda \rightarrow \infty$, which consists of eigenfunctions such that
\[
\sup_{B(y,r) \in \mathcal B_{\lambda}} \left|
\frac 1{\vol(B(y,r))} \int_{B(y,r)} |\psi_n(x)|^2 \, \dvol(x)-1 \right|=o(1) ,\quad \lambda \rightarrow \infty, \;\psi_n\in S_\lambda\;.
\] 
\end{theorem}
 Theorem~\ref{thm:3dloc} reaches the same scale $r>\lambda^{\frac{-1}{2(d-1)}+o(1)}$  as Theorem \ref{thm:small balls}.  
Moreover, it gives that the $L^2$ mass
of ``almost all" eigenfunctions in the $\lambda$-eigenspace equidistributes inside balls
of radii $r> \lambda^{\frac{-1}{2(d-1)}+o(1)}$
 whereas Theorem \ref{thm:small balls}
does not guarantee the existence of even one such eigenfunction. 
We believe that the analogue of Theorem~\ref{thm:3dloc} also holds in 
dimensions $d\geq 5$.

\begin{remark}\label{rem:Bourgain}
Bourgain (private communication) has pointed out that our result is sharp, in that for $d\geq 3$,  under the conditions on $\lambda$ of Theorem~\ref{thm:3dloc}, for radii $\lambda^{-1/2}<r<\lambda^{-\frac 1{2(d-1)}-o(1)}$ each $\lambda$-eigenspace has an  ONB for which a positive proportion of the eigenfunctions fail to equidistribute in $B(0,r)$, in fact for which 
$$\frac 1{\vol(B(0,r))} \int_{B(0,r)} |\psi_n(x)|^2 \, \dvol(x) \sim 0, \quad \lambda\to \infty\;.$$
 The construction is detailed in \S~\ref{sec:below}. 
\end{remark}

\subsection{Discrepancy} \label{subsec:disc}
Given an ONB $\{ \psi_n\}$ consisting of eigenfunctions of the Laplacian,  and  $a\in
C^\infty(\T^d)$, let
\begin{equation*}
V_2(a,\Lambda):=\frac 1{\#\{\lambda_n\leq \Lambda\}}\sum_{\lambda_n\leq \Lambda} \left|
\int_{\T^d} a(x) |\psi_n(x)|^2 \dvol(x)- 
  \int_{\T^d} a(x) \dvol(x)  \right|^2.
\end{equation*}
Here $\dvol(x)=dx/(2\pi)^d$ where $dx$ is Lebesgue measure.
 Marklof and Rudnick \cite{MR} showed
that $V_2(a,\Lambda)$
decays as $\Lambda\to \infty$. 
 This was done via arguing as in Schnirelman's theorem, and using Kronecker's theorem that generic
geodesics are uniformly distributed when projected to configuration
space; the point of \cite{MR} was that this argument extends to
rational polygons. Hezari and Rivi\`ere \cite{HRtor} arrive at their results on small scale equidistribution
  by  giving a quantitative rate of decay of $V_2(a,\Lambda)$.

We will derive Theorem~\ref{thm:small balls} from an upper bound on the $L^1$ discrepancy 
\begin{equation*}
V_1(a,\Lambda):=\frac 1{\#\{\lambda_n\leq \Lambda\}}\sum_{\lambda_n\leq \Lambda} \left|
\int_{\T^d}a(x)|\psi_n(x)|^2  \dvol(x)
- \int_{\T^d}a(x) \dvol(x)  \right|.
\end{equation*}   
For a fixed a trigonometric polynomial, we will show  that
\begin{equation}\label{bd on L1 dis}
V_1(a ,\Lambda)\ll_{a } \Lambda^{-1/2 } \;.
\end{equation}

Note that for chaotic systems, it is expected that the   $L^1$ discrepancy $V_1(a,\Lambda)$ is larger, of size about $\Lambda^{-1/4}$, see \cite{EFKAMM, FP} giving physical arguments for generic chaotic systems, and \cite{LS, KR} for rigorous results of this quality for the $L^2$ discrepancy in arithmetic settings, and \cite{Zelditchupper} for logarithmic upper bounds for the general negatively curved case (see also \cite{Schubert}).

\subsection{About the proofs} 
Our arguments rely upon lattice point
estimates in place dynamical properties of the geodesic flow.
In particular, the proof of the bound \eqref{bd on L1 dis}, given in Section~\ref{sec:small1},  combines harmonic analysis and a lattice point argument from the geometry of numbers (see Lemma  \ref{lem:geom numbers}).
The proof of Theorem~\ref{thm:3dloc} in Section~\ref{sec:indspace} replaces this lattice point count with 
a more refined statistic, which counts lattice points on a sphere
which lie within a small spherical cap (see Remark \ref{rem:caps}). 
To estimate this quantity, we require
deeper arithmetic information 
on the number of representations of a  positive definite binary quadratic form by sums of squares of  linear forms.  This is also used in the construction of ``massive irregularities" in high dimensions in \S~\ref{sec:deep}. Bourgain's argument, which shows Theorem \ref{thm:small balls} reaches
the optimal scale, is detailed in \S~\ref{sec:below} and   also relies upon estimates for
the number lattice points within spherical caps.   
The construction of quantum irregularities in \S~\ref{sec:scar2} relies on more direct arguments.

\subsection{Notation}  
Throughout we use the notation,
$f(x) \ll g(x)$, by which we mean
$f(x)=O(g(x))$. In addition we write
$f(x) \gg g(x)$ provided there exists a $c>0$ such that
$|f(x)| \ge c g(x)$ for all $x$ under consideration, and, if $f(x) \ll g(x)$ and
$f(x) \gg g(x)$ we write $f(x)\approx g(x)$. 

 \subsection{Acknowledgments}
 We thank Jean Bourgain, Jon Keating, St\'ephane Nonnenmacher, Hamid Hezari, Gabriel Rivi\`ere, and Suresh Eswarathasan for their comments on an early version of the paper.  
The research leading to these results has received funding from the European
Research Council under the European Union's Seventh Framework
Programme (FP7/2007-2013) / ERC grant agreement n$^{\text{o}}$
320755.

\section{Small scale equidistribution}
\subsection{$L^1$ discrepancy on the torus} \label{sec:small1}

The goal of this section is to prove the upper bound  \eqref{bd on L1 dis} on the $L^1$ discrepancy.

On the torus $\mathbb T^d$ each eigenfunction $\psi_n$ of $-\Delta$ with eigenvalue $\lambda_n$ is of the following form
$$\psi_n=\sum_{\mu \in \mathbb Z^d : |\mu|^2=\lambda_n}c_n(\mu)e_\mu$$ 
where $e_\mu(x) = e^{\i \langle \mu, x\rangle}$. Throughout, we assume $\psi_n$ is $L^2$-normalized so that
\begin{equation} \notag
 \int_{\T^d} |\psi_n(x)|^2\dvol(x) =\sum_{|\mu|^2=\lambda_n}|c_n(\mu)|^2=1\;.
 \end{equation}

\begin{lemma}\label{two ONB}
For $\mu \in \mathbb Z^d$ such that $|\mu|^2=\lambda$ we have
$$\sum_{\lambda_n=\lambda} |c_n(\mu)|^2 = 1\;.$$
\end{lemma}
\begin{proof}
The functions $\{\psi_n:\lambda_n=\lambda\}$ and $\{e_\mu:|\mu|^2=\lambda\}$ are both orthonormal bases of the $\lambda$-eigenspace of $-\Delta$, with respect to the inner product 
$$ \langle f,g \rangle = \int_{\T^d}f(x)\overline{g(x)}\dvol(x) \;.$$
Hence in the expansion
$$ \psi_n = \sum_{|\mu|^2=\lambda} \langle \psi_n,e_\mu\rangle e_\mu$$
we have 
$$ \langle \psi_n,e_\mu\rangle= c_n(\mu) $$
and hence  the expansion of $e_\mu$ is  
$$ e_\mu  = \sum_{\lambda_n=\lambda} \langle e_\mu,\psi_n\rangle \psi_n =  \sum_{\lambda_n=\lambda} \overline{c_n(\mu)}\psi_n$$
and therefore 
$$
\sum_{\lambda_n=\lambda} |c_n(\mu)|^2 = \sum_{\lambda_n=\lambda} | \langle e_\mu,\psi_n\rangle |^2 = \langle e_\mu,e_\mu \rangle = 1\;.
$$
\end{proof}

\begin{lemma}\label{lem:eigenspace} 
If $|\zeta|\leq 2\sqrt{\lambda}$ then 
$$\sum_{\lambda_n=\lambda} \left| \int_{\mathbb T^d} e_\zeta(x)|\psi_n(x)|^2 \dvol(x) \right| \leq \#\{\mu\in \Z^d:|\mu|^2=\lambda=|\mu+\zeta|^2\} \;.
$$
If $|\zeta|>2\sqrt{\lambda}$ then each summand is zero. 
\end{lemma}
\begin{proof}
Expand $\psi_n$ to get  
\begin{equation*}
\begin{split}
\left| \int_{\TT^d} e_\zeta(x)|\psi_n(x)|^2 \dvol(x) \right| &= \left|
\sum_{|\mu|^2=\lambda_n=|\mu+\zeta|^2} c_n(\mu)\overline{c_n(\mu+\zeta)} \right| \\
&\leq \sum_{|\mu|^2=\lambda_n=|\mu+\zeta|^2} \frac 12|c_n(\mu)|^2 + \frac 12 |c_n(\mu+\zeta)|^2.
\end{split}
\end{equation*}
Hence,
\begin{equation*}
\begin{split}
\sum_{\lambda_n=\lambda} \left|\int_{\TT^d} e_\zeta(x)|\psi_n(x)|^2 \dvol(x)\right|
&\leq \sum_{\lambda_n=\lambda} \sum_{|\mu|^2=\lambda_n=|\mu+\zeta|^2} \frac 12|c_n(\mu)|^2 + \frac 12 |c_n(\mu+\zeta)|^2\\
&= \sum_{|\mu|^2=\lambda=|\mu+\zeta|^2} \frac 12 \sum_{\lambda_n=\lambda} |c_n(\mu)|^2 + \frac 12 \sum_{\lambda_n=\lambda} |c_n(\mu+\zeta)|^2\\
&= \sum_{|\mu|^2=\lambda=|\mu+\zeta|^2} \frac 12 + \frac 12
\end{split}
\end{equation*}
since  by Lemma~\ref{two ONB} both inner sums equal one. Hence, 
$$
\sum_{\lambda_n=\lambda} \left|\int_{\TT^d} e_\zeta(x)|\psi_n(x)|^2 \dvol(x)\right|
\leq\#\{\mu\in \Z^d:|\mu|^2=\lambda=|\mu+\zeta|^2\}\;.
$$ 
Now if $|\zeta|>2\sqrt{\lambda}$ then there is no (real) solution of $|\mu|^2=\lambda=|\mu+\zeta|^2$ and hence all terms above vanish. 
\end{proof}

\begin{lemma}\label{lem:geom numbers}
For a nonzero integer vector $\zeta\in \Z^d$ write $\zeta = m\widehat \zeta$, with $m\geq 1$ and $\widehat\zeta$ a primitive integer vector. 
If $0<|\zeta|\leq 2\sqrt{X}$  then 
$$\#\{\mu\in \Z^d: |\mu|^2\leq X, |\mu|^2=|\mu+\zeta|^2 \} \ll \frac{(\sqrt{X})^{d-1}} 
{|\widehat\zeta|} \; ,
$$
while for $|\zeta|>2\sqrt{X}$, the set above is empty. 
\end{lemma}
\begin{proof}
Suppose we have a solution $\mu \in \mathbb Z^d$ with $|\mu+\zeta|=|\mu| \le \sqrt{X}$ then
 $|\zeta|\leq |\mu+\zeta|+|\mu| \leq 2\sqrt{X}$ and hence if $|\zeta|>2\sqrt{X}$ then there are no solutions. So from now on assume $|\zeta|\leq 2\sqrt{X}$.  

The equality $|\mu|^2 = |\mu+\zeta|^2$ is equivalent to 
\begin{equation} \label{eq:equalityequiv}
2\langle \mu,\zeta \rangle = - |\zeta|^2
\end{equation}
which only has solutions if $|\zeta|^2$ is even. 

If there are no solutions to \eqref{eq:equalityequiv} with $|\mu|\leq \sqrt{X}$, then we are done. Otherwise, there exists a solution 
$\mu_0$ and any other such solution satisfies 
\begin{equation*}
\langle \mu-\mu_0,\zeta \rangle =0, \quad |\mu-\mu_0|\leq 2\sqrt{X}\;.
\end{equation*}
We see that the number of solutions $|\mu|\leq \sqrt{X}$ to \eqref{eq:equalityequiv} is bounded by the number of $\nu\in \Z^d$ such that 
\begin{equation*}
\langle \nu, \zeta \rangle = 0,\quad |\nu|\leq 2\sqrt{X}\;.
\end{equation*}
That is, we are counting lattice points in the $(d-1)$-dimensional sub-lattice which is the ortho-complement of $\zeta$,  which lie in a ball. The co-volume (discriminant) of this sub-lattice 
is $|\widehat\zeta|$, where $\zeta = m\widehat \zeta$ with $\widehat \zeta$ primitive, $m\geq 1$ integer, and by \cite[Section 2]{Sc}  the number of such integer solutions is 
\begin{equation*} \label{eq:latticeptcount}
 c_d \frac{(2\sqrt{X})^{d-1} }{|\widehat\zeta|} + O((\sqrt{X})^{d-2})\ll \frac{X^{(d-1)/2}}
{|\widehat\zeta|},
\end{equation*}
since $|\widehat\zeta|\leq |\zeta|\ll \sqrt{X}$.  Here $c_d=\frac{\pi^{d/2}}{\Gamma(\frac{d}{2}+1)}$ is the volume of the $d$-dimensional unit ball in $\mathbb R^d$. 
\end{proof}

\begin{proposition} \label{prop:upperbd}
If $|\zeta|\leq 2\sqrt{\Lambda}$ then 
$$V_1(e_\zeta,\Lambda)\ll \frac{\Lambda^{-1/2}}{|\widehat \zeta|}\;.
$$
If $|\zeta|>2\sqrt{\Lambda}$ then $V_1(e_\zeta,\Lambda)=0$. 
\end{proposition}

\begin{proof}
By Lemma~\ref{lem:eigenspace}, 
\begin{equation*}
\begin{split}
V_1(e_\zeta,\Lambda) &= \frac 1{\#\{\lambda_n\leq \Lambda\}}\sum_{\lambda\leq \Lambda} \sum_{\lambda_n=\lambda}
 \left| \int_{\mathbb T^d} e_\zeta(x)|\psi_n(x)|^2  \dvol(x) \right| \\
 &  \leq \frac 1{\#\{\lambda_n\leq \Lambda\}}\sum_{\lambda\leq \Lambda}\#\{\mu\in \Z^d:|\mu|^2=\lambda=|\mu+\zeta|^2\} \\
 &= 
 \frac 1{\#\{\lambda_n\leq \Lambda\}}\#\{|\mu|\leq \sqrt{\Lambda}:|\mu|^2 = |\mu+\zeta|^2\}\;.
\end{split}
\end{equation*} 
The denominator is $\#\{\lambda_n\leq \Lambda\}\approx \Lambda^{d/2}$ (Weyl's law, which follows from an elementary argument since $\#\{\lambda_n\leq \Lambda\}=\#\{ \mu \in \mathbb Z^d : |\mu|^2 \le \Lambda \}$),  
 while by Lemma~\ref{lem:geom numbers}, the numerator is $\ll (\sqrt{\Lambda})^{d-1}/|\widehat \zeta|$, which gives the claim. 
\end{proof}
Note that the upper bound  \eqref{bd on L1 dis} on the $L^1$ discrepancy  $V_1(a,\Lambda) $ for a general trigonometric polynomial follows from Proposition~\ref{prop:upperbd}.

\subsection{Proof of Theorem~\ref{thm:small balls}}

We will need majorants and minorants for the indicator function of a ball. 
We now cite Lemma 4 of Harman \cite{Harman} (see also the work of Holt \cite{Holt} and Holt and Vaaler \cite{HoltVaaler}), which constructs an appropriate version of Beurling-Selberg polynomials:
\begin{lemma} \label{lem:BSfns}
Let $B(0,r)\subset \T^d$ be the ball of radius $r$ around the origin. Let $T,r>0$ with $Tr \ge 1$.
There exist trigonometric polynomials
$a^{\pm}$ such that:
\begin{itemize}
\item[i)] $a^{-}(x) \le \mathbf 1_{B(0,r)}(x) \le a^{+}(x)$;
\item[ii)] $ \widehat a^{\pm}(\zeta)=0$ if $|\zeta|\ge T$;
\item[iii)] $\widehat a^{\pm}(0)=\tmop{vol}(B_d(0,r))+O_d(r^{d-1}/T)$;
\item[iv)] $|\widehat a(\zeta)| \ll_d r^d$.
\end{itemize}
\end{lemma}
\begin{proof}[Proof of Theorem \ref{thm:small balls}]
Let 
$$
\mathcal B_n=\left\{ B(y,r) \subset \mathbb T^d : r> \lambda_n^{-\theta_1}\right\}
$$
with $\theta_1$ to be determined later. Also, 
for $r> \lambda_n^{-\theta_1}$
let $a_n^{\pm}$ be the Beurling-Selberg polynomials from Lemma~\ref{lem:BSfns},
which majorize and minorize the indicator function of the ball $B(0,r)$ and also satisfy $\widehat a_n^{\pm}(\zeta)=0$
for $|\zeta| \ge T_n=\lambda_n^{\theta_2}$,
with $\theta_2>\theta_1$. The trigonometric polynomials
$$b_{n,y}^\pm(x):=a_n^\pm(x-y) $$
 majorize and minorize the translated ball $B(y,r) = y+B(0,r)$,  and their Fourier coefficients are given by 
$\widehat b_{n,y}^\pm(\zeta) = e^{-i\zeta \cdot y}\widehat a_{n}^\pm(\zeta)$,  which therefore satisfy the same inequalities as 
$\widehat a_{n}^\pm(\zeta)$ in Lemma~\ref{lem:BSfns} (independently of $y$). In particular, $|\widehat b_{n,y}^{\pm}(\zeta)|
\ll r_n^d$ and for $T_n=\lambda_n^{\theta_2}$ with $\theta_2>\theta_1 \ge 0$ it follows that $\widehat b_{n,y}^{\pm}(0)=\tmop{vol}(B(0,r))(1+O(\lambda_n^{\theta_1-\theta_2}))$.

For $\delta >0$ let
\[
\mathcal S^{\pm}=\bigg\{ \lambda_n  : \sup_{ B(y,r) \in \mathcal B_n} \bigg|\frac{\int_{\mathbb{T}^d} b_{n,y}^{\pm}(x) |\psi_n(x)|^2 \, \dvol(x)}{\int_{\mathbb{T}^d} b_{n,y}^{\pm}(x) \, \, \dvol(x)} -1\bigg|  \ge \lambda_n^{-\delta} \bigg\}.
\]
We  will now show that for $\theta_1 <\theta_2< \frac{1}{2(d-1)}-\delta$  the sets $\mathcal S^{\pm}$ have zero density.
First note that
\[
\begin{split}
\sup_{ B(y,r) \in \mathcal B_n}
\bigg|\frac{\int_{\mathbb{T}^d} b_{n,y}^{\pm}(x) |\psi_n(x)|^2 \, \dvol(x)}{\int_{\mathbb{T}^d} b_{n,y}^{\pm}(x) \, \, \dvol(x)} -1\bigg| \le & \sum_{1 \le |\zeta|\le \lambda_n^{\theta_2}} \sup_{ B(y,r) \in \mathcal B_n}\left|\frac{\widehat b_{n,y}^{\pm}(\zeta)}{\widehat b_{n,y}^{\pm}(0) } \langle e_{\zeta} \psi_n,\psi_n \rangle \right| \\
\ll & \sum_{1 \le |\zeta|\le \lambda_n^{\theta_2}}\left|\langle e_{\zeta} \psi_n,\psi_n \rangle \right| \; .
\end{split}
\]
Next, apply Chebyshev's inequality, the above estimate, and Lemma \ref{lem:eigenspace} to get that
\[
\begin{split}
\frac{\# \{ \lambda_n \in \mathcal S^{\pm} : \lambda_n \le \Lambda\}}{\#\{\lambda_n\leq \Lambda\}}
  \ll &	\frac{1}{ \Lambda^{d/2-\delta}} \sum_{\lambda_n \le \Lambda} \sup_{ B(y,r) \in \mathcal B_n}
\bigg|\frac{\int_{\mathbb{T}^d} b_{n,y}^{\pm}(x) |\psi_n(x)|^2 \, \dvol(x)}{\int_{\mathbb{T}^d} b_{n,y}^{\pm}(x) \, \, \dvol(x)} -1\bigg| \\
	\ll& \frac{1}{\Lambda^{d/2-\delta}} \sum_{\lambda \le \Lambda} \sum_{1 \le |\zeta|\le \lambda^{\theta_2}} \sum_{\lambda_n=\lambda} \left|  \langle e_{\zeta} \psi_n,\psi_n \rangle \right| \\
	\ll & \frac{1}{ \Lambda^{d/2-\delta}} \sum_{\lambda \le \Lambda} \sum_{1 \le |\zeta|\le \lambda^{\theta_2}}  \#\{\mu\in \Z^d:|\mu|^2=\lambda=|\mu+\zeta|^2\}. 
\end{split}
\] 
By Lemma \ref{lem:geom numbers} 
\[
\sum_{\lambda \le \Lambda} \sum_{1 \le |\zeta|\le \lambda^{\theta_2}}  \#\{\mu\in \Z^d:|\mu|^2=\lambda=|\mu+\zeta|^2\} \ll \Lambda^{(d-1)/2} \sum_{1\le |\zeta|\le \Lambda^{\theta_2}} \frac{1}{|\widehat \zeta|},
\]
where $\zeta=m\widehat \zeta$ and $\widehat \zeta$ is primitive.
The last sum is bounded by
\begin{equation} \notag
\sum_{1 \le m \le \Lambda^{\theta_2}}  \sum_{1 \le | \widehat \zeta| \le \Lambda^{\theta_2}/m} \frac{1}{|\widehat \zeta|}
\ll  \Lambda^{\theta_2(d-1)} \sum_{1 \le m \le \Lambda^{\theta_2}} \frac{1}{m^{d-1}} \ll
\begin{cases}
 \Lambda^{\theta_2} \log \Lambda \text{ if } d=2, \\
 \Lambda^{\theta_2(d-1)} \text{ if } d \ge 3.
\end{cases}
\end{equation}
Collecting estimates, we have shown that
\[
\frac{\# \{ \lambda_n \in \mathcal S^{\pm} : \lambda_n \le \Lambda\}}{\#\{\lambda_n\leq \Lambda\}}
  \ll 
\Lambda^{\theta_2(d-1)-\frac12+\delta} \log \Lambda
\]
which tends to zero for $\theta_2<\frac{1}{2(d-1)}-\delta$.

To conclude the proof first observe that
if $\lambda_n \notin \mathcal S^+$ with  $\theta_1<\theta_2<\frac{1}{2(d-1)}-\delta$ (so $\lambda_n$ lies in a set of density one) it follows by
 parts $i)$ and $iii)$ of Lemma \ref{lem:BSfns} that
\begin{equation} \label{eq:eqfin1}
\begin{split}
&\int_{B(y,r)} |\psi_n(x)|^2 \dvol(x)-\vol(B(y,r)) \\
& \qquad
\le \int_{\mathbb{T}^d} b_{n,y}^{+}(x) |\psi_n(x)|^2 \, \dvol(x)-\widehat b_{n,y}^+(0)+O(r^d \lambda_n^{\theta_1-\theta_2})\;.
\end{split}
\end{equation}
A similar analysis holds for $\lambda_n \notin \mathcal S^{-}$ with the inequality reversed. Therefore, for $\lambda_n \notin (\mathcal S^{+} \cup \mathcal S^{-})$ and $\theta_1< \theta_2<\frac{1}{2(d-1)}-\delta$ (so $\lambda_n$
 lies in a density one set)
\begin{equation} \label{eq:eqfin2}
\begin{split}
& \sup_{B(y,r) \in \mathcal B_n} \left| \int_{B(y,r)} |\psi_n(x)|^2 \, \dvol(x)- \vol(B(y,r)) \right|  \\
 & \qquad 
\le \max_{\pm} \sup_{ B(y,r) \in \mathcal B_n}
\bigg|\int_{\mathbb{T}^d} b_{n,y}^{\pm}(x) |\psi_n(x)|^2 \, \dvol(x) -\widehat b_{n,y}^{\pm}(0)\bigg|+O(r^d\lambda_n^{\theta_1-\theta_2})  \\
& \qquad \ll r^d \lambda^{-\delta}+ r^d \lambda_n^{\theta_1-\theta_2}
 \; ,
\end{split}
\end{equation}
so the claim follows. 
\end{proof}

\section{Irregularities of quantum equidistribution} \label{sec:scar2}

In the previous section we saw that given an ONB of eigenfunctions $\{\psi_n\}$ the $L^2$ mass of almost all eigenfunctions $\psi_n$ equidistributes within balls with radii $r_n \ge \lambda_n^{-\frac{1}{2(d-1)}+o(1)}$.
We will show the existence of a sequence of eigenvalues $\{\lambda_m\}$ which tends to infinity with corresponding 
eigenfunctions whose $L^2$ mass is not equidistributed within balls with radii $r_m \ge \lambda_m^{-1/2+o(1)}$, which is just above the Planck scale.

\begin{theorem} \label{prop:scar2}
There exists a sequence $\{ \lambda_m \}_m$ of eigenvalues of $-\Delta$ on $\mathbb T^d$ with $\lambda_m\to \infty$ and corresponding $L^2$-normalized eigenfunctions $\psi_m$ so that for any choice of radii $r_m$ so that $r_m\to 0$, but $r_m\sqrt{\lambda_m}\to \infty$, 
$$
\frac 1{\tmop{vol}(B(0,r_m))}\int_{B(0,r_m)}|\psi_n( x)|^2 \tmop{dvol}(x) = 2 +o(1) \qquad (m \rightarrow \infty) \;.
$$
\end{theorem}

\begin{proof}
Let 
$\lambda_m = m^2+(m+1)^2$ and take 
$$\psi_m(x) = \cos(mx_1+(m+1)x_2) + \cos((m+1)x_1+mx_2), $$
where $x = (x_1,x_2,\dots, x_d)$, which are $L^2$-normalized eigenfunctions  on $(\TT^d, \dvol)$ with eigenvalue $\lambda_m$. 
See Figure~\ref{fig:scar2} for a plot of the intensities  $|\psi_m(x)|^2$.  
\begin{figure}
\begin{center}
  \includegraphics[width=50mm]{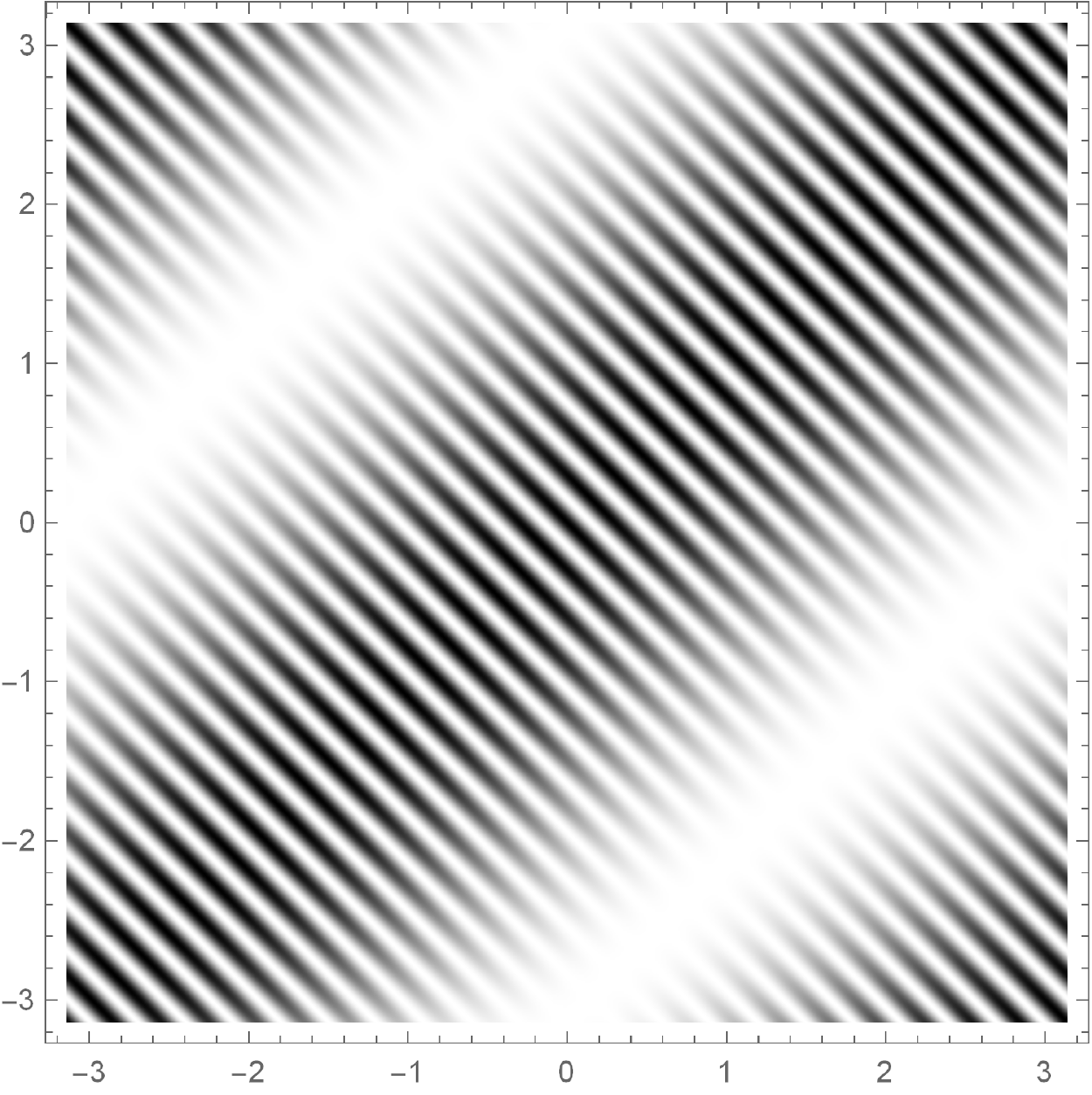} 
 \caption{Plot of  the intensities of $|\psi_m(x)|^2$  for $m=10$ in dimension $d=2$, where $\psi_m(x) = \cos(mx_1+(m+1)x_2) + \cos((m+1)x_1+mx_2).$}\label{fig:scar2}
\end{center}
\end{figure}

Squaring out we get 
\begin{multline*}
|\psi_m(x)|^2 = 1 + \cos(x_1-x_2) + \cos((2m+1)(x_1+x_2)) \\+ \frac 12\cos((2m+2)x_1+2mx_2) +\frac 12\cos (2mx_1+(2m+2)x_2)  
\end{multline*}
 and we wish to average this over the ball $B(0,r_m)$. 

For the term $\cos(x_1-x_2)$, observe that its average over $B(0,r_m)$ tends to $1$, because on this shrinking ball, we have  
$|x_1-x_2|\leq 2r_m$ and hence $\cos(x_1-x_2) = 1+ O(r_m^2)$, so that 
$$ 
\frac 1{\vol(B(0,r_m))}\int_{B(0,r_m)} \cos(x_1-x_2)\tmop{dvol}(x) = 1+O(r_m^2) \to 1, \quad {\rm as}\; r_m\to 0 \;.
$$

To handle the other three terms, note that if $\mu\in \Z^d$ is any frequency vector, then changing variables we find
$$\
\frac 1{\vol (B(0,r_m))}\int_{B(0,r_m)} \cos(\langle \mu, x\rangle) \tmop{dvol}(x) = \frac 1{\vol(B(0,1))}\int_{B(0,1)}\cos(\langle r_m \mu, y\rangle) \tmop{dvol}(y)
$$
that is we get the Fourier transform of the unit ball at the frequency $r_m\mu$. As is well known, the Fourier transform of the unit ball decays in {\em all directions}: 
$$
 \frac 1{\vol( B(0,1))}\int_{B(0,1)}\cos(\langle \xi, x\rangle) \tmop{dvol}(x) \to 0,\quad {\rm as}\; |\xi|\to \infty \;.
$$
 Therefore, applying this to the frequency vectors  $\mu = (2m,2m+2,\vec 0)$, $(2m+2,2m,\vec 0)$ and $(2m+1,2m+1,\vec 0)$, which have length $|\mu|\approx m$,  we get
$$
\frac 1{\vol(B(0,r_m))}\int_{B(0,r_m)} \cos(\langle \mu, x\rangle) \tmop{dvol}(x) \to  0,\quad r_m|\mu|\to \infty \;.
$$
Thus whenever $r_m\to 0$, with $r_m\cdot m \approx  r_m\sqrt{\lambda_m}\to \infty$, 
$$
\frac 1{\vol(B(0,r_m))} \int_{B(0,r_m)} |\psi_m(x)|^2  \tmop{dvol}(x) = 2+o(1)
$$
giving our claim. 
\end{proof}

\section{Below the critical radius: $r<\lambda^{-\frac 1{2(d-1)}}$}\label{sec:below}

In this section, we detail Bourgain's argument which gives for balls with radii $r<\lambda^{-\frac 1{2(d-1)}-o(1)}$  that in each eigenspace there is an orthonormal set of eigenfunctions with size exceeding a positive multiple of the dimension of the eigenspace, which consists of eigenfunctions whose $L^2$ mass is scarce in the ball $B(0,r)$.

Denote 
$$ \mathcal E_{\lambda}= \{\mu\in \Z^d:|\mu|^2=\lambda\},\qquad  N_\lambda = \#\mathcal E_{\lambda} \;.
$$
\begin{theorem}[Bourgain] \label{thm:bourgain}
Suppose $d \ge 3$. Also, if $d=3$ suppose
$\lambda \not\equiv 0,4,7 \mod 8$, and, if $d=4$
suppose $\lambda$ is odd.
Then for each such $\lambda$-eigenspace there exists an orthonormal set of eigenfunctions
$\mathcal A \subset \{ \psi_{\lambda_n}\}_{\lambda_n=\lambda}$ with size $\#\mathcal A \gg N_{\lambda}$ such that 
for each $\psi \in \mathcal A$
\[
\frac{1}{\tmop{vol}(B(0,r))} \int_{B(0,r)} |\psi(x)|^2 \tmop{dvol}(x) \rightarrow 0 \qquad (\lambda \rightarrow \infty)
\]
provided that $r<\lambda^{-\frac{1}{2(d-1)}-o(1)}$.
\end{theorem}

Completing the orthonormal set $\mathcal A$ in Theorem \ref{thm:bourgain} (in any way) gives an ONB of eigenfunctions $\mathcal B$ with the property that a positive proportion of $\psi \in \mathcal B$ do not equidistribute within the small balls $B(0,r)$, $r<\lambda^{\frac 1{2(d-1)} -o(1)}$. Hence, the scale achieved in Theorems \ref{thm:small balls} and \ref{thm:3dloc} is sharp.

 Before detailing Bourgain's argument we require the following lemma on the distribution of 
points on spheres.
For each point $\mu\in \sqrt{\lambda}S^{d-1}$, we associate the cap $\dcap(\mu;Y)=\tmop{Ball}(\mu,Y)\cap \sqrt{\lambda}S^{d-1}$ of size $Y$ about $\mu$, where $\tmop{Ball}(x,Y)=\{y \in \mathbb R^d : |x-y| \le Y\}$.

\begin{lemma} \label{lem:latticecap}
Suppose for a sequence of $\lambda$'s, we are given a finite set $\mathcal A_\lambda \subset \sqrt{\lambda}S^{d-1}$  of points on the sphere, with cardinality $\#\mathcal A_\lambda\to \infty$ as $\lambda\to \infty$. Let $Y=Y_\lambda$ satisfy $Y_\lambda\gg \lambda^{1/2+o(1)}/(\#\mathcal A_\lambda)^{\frac 1{d-1}}$. 
Then
the set $\mathcal V \subset \mathcal A_{\lambda}$ consisting  of $\nu \in \mathcal A_{\lambda}$ such that 
\[
\#\left( \mathcal A_{\lambda} \cap \tmop{cap}(\nu,Y)\right) \ge 2
\] 
has density one: $\# \mathcal V \sim \#\mathcal A_{\lambda}$ as $\lambda\to \infty$.
\end{lemma}

\begin{proof}
%
Let 
\[
\mathcal U=\left\{ \mu \in \mathcal A_{\lambda}: \#(\mathcal A_{\lambda}\cap \dcap(\mu;Y ))<2 \right\}.
\]
We wish to show that   $\#\mathcal U =o(\#\mathcal A_{\lambda})$. 

Each point on the sphere $\sqrt{\lambda}S^{d-1}$ is contained in at most one of the caps of size $Y/2$ around $\mu \in \mathcal U$, because if $ \dcap(\mu_1;Y/2)\cap \dcap(\mu_2;Y/2)$ is non-empty for distinct $\mu_1\neq \mu_2 \in \mathcal U$ then  $\mu_2\in \dcap(\mu_1;Y)$ contradicting the assumption $\mu_2 \in \mathcal U$.  Consequently the caps $\dcap(\mu_1;Y/2)$ and $\dcap(\mu_1;Y/2)$
are disjoint, so that we have 
$$ \vol \Big(\bigcup\limits_{\mu \in \mathcal U} \dcap(\mu,Y/2) \Big) = \sum_{\mu \in \mathcal U} \vol\Big(\dcap(\mu,Y/2)\Big) \approx Y^{d-1} \#\mathcal U
$$
and also we have the trivial bound
$$
 \vol \Big(\bigcup\limits_{\mu \in \mathcal U} \dcap(\mu,Y/2) \Big) \leq \vol(\sqrt{\lambda}S^{d-1}) \ll_d \lambda^{(d-1)/2}\;.
$$
Combining these formulas   we obtain
$$ \frac{\# \mathcal U}{\#\mathcal A_{\lambda}} \ll  \frac{1}{\#\mathcal A_{\lambda}} \cdot \frac{\lambda^{(d-1)/2}}{Y^{d-1}}
\ll \lambda^{-o(1)}\to 0$$
under our assumption on $Y$, which gives the claim.
\end{proof}

\begin{proof}[Proof of Theorem \ref{thm:bourgain}]
First observe that if we have two distinct lattice points 
$\mu\neq \mu'\in \mathcal E_{\lambda}$, which are \underline{close}: $0<|\mu-\mu'|<M_{\lambda}$ (we will take $M_{\lambda}=\lambda^{\frac{1}{2(d-1)}+o(1)}$), then the eigenfunction  
$$\psi_\mu(x):=\frac 1{\sqrt{2}} \Big(e^{i\langle \mu,x \rangle} - e^{i\langle \mu',x \rangle} \Big)
$$
fails to equidistribute in the ball $B(0,r)$ centered at the origin for any  $r=o(M_{\lambda}^{-1})$. Indeed, for $x\in B(0,r)$ 
$$|\psi_\mu(x)|^2 = 1-\cos(\langle \mu-\mu',x\rangle) = O\Big((r|\mu-\mu'|)^2\Big)
$$
and since $r|\mu-\mu'|  \leq rM_{\lambda} = o(1)$, we have  
$$|\psi_\mu(x)|^2 = o(1), \quad x\in B(0,r)\;.
$$
Therefore
$$\frac 1{\vol (B(0,r))} \int_{B(0,r)}|\psi_\mu(x)|^2\dvol(x) \to 0\;.$$

 Next we claim that there is a set $\mathcal S\subset \mathcal E_{\lambda}$ containing a positive  proportion of $\mu$'s ($\#\mathcal S/N_{\lambda} \gg 1$) such that :
\begin{itemize}
\item  for each $\mu\in \mathcal S$ there is another 
lattice point $\mu'$ which is close to $\mu$: $|\mu-\mu'|< \lambda^{\frac 1{2(d-1)}+o(1)}$; 
\item if $\mu\neq\nu\in \mathcal S$ are distinct, then the pairs $\{\mu,\mu'\}$ and $\{\nu,\nu'\}$ are disjoint, that  is $\nu\neq \mu'$ and $\nu'\neq \mu,\mu'$.  
\end{itemize}
Given this, we form for each $\mu\in \mathcal S$ the eigenfunction $\psi_\mu$, and then for $\mu\neq \nu\in \mathcal S$ the pairs $\{\mu,\mu'\}$ and $\{\nu,\nu'\}$ are \underline{disjoint}, and so the  eigenfunctions $\psi_\mu$ and $\psi_\nu$ are orthogonal. This establishes Bourgain's result Theorem \ref{thm:bourgain}

 It remains to prove the claim. Let $Y_\lambda = \lambda^{\frac 1{2(d-1)}+o(1)}$. We construct $\mathcal S$ as follows: In Lemma~\ref{lem:latticecap} first take $\mathcal A_\lambda^0 = \vE_\lambda$, 
and note that under the assumption of the theorem on $\lambda$, we have $\#\mathcal A_\lambda^0=N_\lambda\gg \lambda^{\frac d2-1-o(1)}$  (see Section~\ref{sec:arith}) so that $Y_{\lambda} \gg \lambda^{1/2+o(1)}/N_\lambda^{\frac 1{d-1}}$. Hence by   Lemma~\ref{lem:latticecap} we get a set $\mathcal V$ of density one.  Take some  $\mu \in \mathcal V$; then there exists $\mu'\in \mathcal A_\lambda^0$ such that $0< |\mu-\mu'| < \lambda^{\frac 1{2(d-1)}+o(1)}$. Now  remove the pair $\{\mu,\mu'\}$ from $\mathcal E_{\lambda}$, to obtain a smaller set  $\mathcal A_\lambda^1 = \mathcal A_\lambda^0\backslash \{\mu,\mu'\}$, and repeat this process $(\frac 12-o(1))N_\lambda$ times, at each time getting a non-empty remainder set $\mathcal A_\lambda^j$, of size $\#\mathcal A_\lambda^j\gg  N_\lambda$, so that still  $Y_\lambda \gg  \lambda^{1/2+o(1)}/\#\mathcal A_\lambda^j$ and we can continue to  invoke Lemma~\ref{lem:latticecap}. 
  
We obtain $(\frac 12-o(1))N_\lambda$ resulting pairs, which by construction  are close and disjoint. In this way  we obtain a set $\mathcal S$ of density $\frac 12 -o(1)$ with the desired properties.  
\end{proof}

\section{Results for individual eigenspaces} \label{sec:indspace}

\subsection{Arithmetic background} \label{sec:arith}
We denote by $R_d(n)$ the number of representations of $n$ as a sum of $d$ squares. This is the dimension of the $n$-eigenspace of the Laplacian on $\T^d$. 
For $d=4$, Jacobi's four square theorem says that $R_4(n) = 8 \sum_{d\mid n, 4\nmid d}d$ so that $R_4(n)\ll n^{1+o(1)}$ and for $n$ odd we have a lower bound $R_4(n)\geq 8n$. 
For $d=3$, we have $R_3(n)\ll n^{1/2+o(1)}$ and Siegel's theorem says that for $n\neq 0,4,7\mod 8$, we have a lower bound $R_3(n)\gg n^{1/2-o(1)}$. When $d \ge 5$, a classical result of Hardy and Ramanujan gives $R_d(n) \approx n^{d/2-1}$.
For more details on these bounds including more precise formulas see e.g.\cite[Chapter 11]{Iwaniec}, \cite{Grosswald}.

For $n,t \ge 1$ let $A_d(n,t)$ denote the number of representations of the positive definite binary quadratic form
\[
Q(x,y)=nx^2+2txy+ny^2
\]
as a sum of squares of $d$ linear forms. That is,
\[
A_d(n,t)=\#\left\{(\mu,\nu)\in \mathbb Z^d \times \mathbb Z^d : \sum_{j=1}^d(\mu_j x+\nu_jy)^2=Q(x,y) \right\}.
\]
where $x$, $y$ are indeterminates.  Equivalently,
\[
A_d(n,t) =\# \left\{ (\mu,\nu) \in \mathbb Z^d \times \mathbb Z^d : |\mu|^2=|\nu|^2=n 
\text{ and } \langle \mu, \nu \rangle=t \right\} \;.
\]

The number of representations of quadratic forms by quadratic forms has been widely studied.
This generalizes the classical problem of representing integers by quadratic forms and for
 a survey of results on these problems see  \cite{SchulzePillot}. The study of the more specific case of representing
a  quadratic form by a sum of squares of linear forms dates back to at least Mordell who studied
the criteria for which such a representation exists in a small number of variables (such a representation always exists
if the number of variables is sufficiently large). In the case $d=3$ 
 Venkov \cite{Venkov} \cite[Chapter 4.16]{VenkovBook} and Pall \cite{Pall42, Pall48} studied $A_3(n,t)$, obtaining an exact, but complicated formula for it. From this one can deduce the following useful bound: 
\begin{lemma}
\label{lem:Abd}
If $|t|<n$ then
\[
A_3(n,t) \ll \gcd(n,t)^{1/2} n^{o(1)}.
\]
\end{lemma}
 This kind of bound was stated and used by Linnik \cite{Linnik40}, who omitted the factor of $\gcd(n,t)^{1/2}$. 
A correct version was given by Pall \cite[\S 7]{Pall42}, \cite[Theorem 4]{Pall48}, 
see also \cite[Proposition 2.2]{BourgainSarnakRudnick}. 

In the case $d=4$, Pall and Taussky \cite{PaTa} established an exact 
formula for $A_4(n,t)$. The relevant case
for us will be when $n$ is odd, in this
case their formulas states the following.

\begin{lemma} \label{lem:A4}
If $n$ is odd and $|t| < n$ then setting $e:=\gcd(n,t)$, we have
\begin{equation*}
A_4(n,t) = \sum_{h|e} R_4(h)
\cdot \# \{ \nu \in \mathbb Z^3 : |\nu|^2=n^2-t^2, \gcd(\nu_1,\nu_2,\nu_3,e)=h \}\;.
\end{equation*}
\end{lemma}
In particular, for $n$ odd Lemma~\ref{lem:A4} gives
\begin{equation}\label{eq:A4id}
A_4(n,t) \ge 8 \cdot R_3(n^2-t^2).
\end{equation}
with equality holding if  $\gcd(n,t)=1$. This is seen by using $R_4(h) \ge 8$ for odd $h$ and noting that  every $\nu$ with $|\nu|^2=n^2-t^2$, will satisfy $\gcd(\nu_1,\nu_2,\nu_3,e)=h$ for some $h\mid e$. 

To get an upper bound for $A_4(n,t)$, 
first note that for $|t|<n$ and $h|e$
\begin{multline*}
\# \{ \nu \in \mathbb Z^3 : |\nu|^2=n^2-t^2, \gcd(\nu_1,\nu_2,\nu_3,e)=h \} \\
 \qquad \qquad \le R_3\left( \frac{n^2-t^2}{h^2}\right) 
\ll  \left(\frac{n^2-t^2}{h^2}\right)^{1/2+o(1)},
\end{multline*}
where in the last step we used the bound $R_3(m) \ll m^{1/2+o(1)}$.
Now use this estimate in Lemma \ref{lem:A4} along with 
the bounds $R_4(h) \ll h^{1+o(1)}$ and $\sum_{h\mid e}1\ll e^{o(1)}$ to get for $n$ odd and $|t|<n$ that
\begin{equation} \label{eq:A4upbd}
A_4(n,t) \ll n^{1/2+o(1)}  (n-t)^{1/2},
\end{equation}
uniformly for $|t|<n$.
\subsection{$L^1$ discrepancy for each $\lambda$-eigenspace}

For $a \in C(\TT^d)$ 
define the localized $L^1$ discrepancy
\[
V_1^{\tmop{loc}}(a,\lambda)=\sum_{\lambda_n =\lambda} |\langle a \psi_n,\psi_n  \rangle-\langle a, 1 \rangle|.
\]

\begin{lemma} \label{lem:localbd}
Suppose $T \le \sqrt{2\lambda}$. Then
\[
\sum_{1 \le |\zeta| \le T}
V_1^{\tmop{loc}}(e_{\zeta},\lambda)
\le \sum_{ \lambda-T^2/2 \le t \le \lambda-1} A_d(\lambda,t).
\]
\end{lemma}

\begin{proof}
Applying Lemma \ref{lem:eigenspace} gives 
 \[
\sum_{1 \le |\zeta| \le T}
V_1^{\tmop{loc}}(e_{\zeta},\lambda)
\le  \sum_{\substack{2 \le |\zeta|^2 \le T^2 }} \# \left\{ \mu : |\mu|^2=\lambda=|\mu+\zeta|^2 \right\}\;.
\]
(Note we can ignore $\zeta$ with $|\zeta|^2$ odd, since for these $\langle e_{\zeta} \psi_n,\psi_n \rangle=0$.)
Next, observe that
\[
\begin{split}
\sum_{\substack{2 \le |\zeta|^2 \le T^2 }} \# \left\{ \mu : |\mu|^2=\lambda=|\mu+\zeta|^2 \right\}
=&   \sum_{2 \le \ell \le T^2}  \sum_{\substack{\mu,\nu \in \mathbb Z^d \\ |\mu|^2=\lambda=|\nu|^2 \\  |\mu-\nu|^2=\ell   }} 1 .
\end{split}
\]
For $|\mu|^2=|\nu|^2=\lambda$ we have $|\mu-\nu|^2=\ell$ iff $\langle \mu, \nu \rangle =(\lambda-\ell/2)$. 
Hence, 
\[
\sum_{2 \le \ell \le T^2}
\sum_{\substack{\mu,\nu \in \mathbb Z^d \\ |\mu|^2=\lambda=|\nu|^2 \\ |\mu-\nu|^2=\ell  }} 1 
= \sum_{ \lambda-T^2/2 \le t \le \lambda-1} A_d(\lambda,t).
\]
\end{proof}

\subsection{Proof of Theorem \ref{thm:3dloc}}
Suppose that $d=3$ or $d=4$.
For $d=3$ suppose
$\lambda \not \equiv 0,4,7 \pmod 8$, so the dimension of the $\lambda$-eigenspace, $N_{\lambda}$,
is  $\approx \lambda^{1/2\pm o(1)}$; if $d=4$
suppose $\lambda$ is odd so that $\lambda\ll N_{\lambda} \ll   \lambda^{1 + o(1)}$.  
Let 
$$
\mathcal B_{\lambda}=\left\{B(y,r) \subset \TT^d
: r \ge \lambda^{-\theta_1} \right\}
$$
for $\theta_1$ to be determined later.
As in the proof of Theorem \ref{thm:small balls} we take $b_{n,y}^{\pm}$ to be Beurling-Selberg polynomials  
which majorize and minorize the indicator function of the ball $B(y,r)$ with $r \ge \lambda_n^{-\theta_1}=\lambda^{-\theta_1}$. 
We take the lengths of the polynomials $b_{n,y}^{\pm}$ to be $T_n=\lambda_n^{\theta_2}=\lambda^{\theta_2}$ with $\theta_2>\theta_1$.
Given an orthonormal basis $ \{\psi_n \}_{\lambda_n=\lambda}$ of the $\lambda$-eigenspace define
\[
\mathcal S_{\lambda}^{\pm}=\left\{ \lambda_n =\lambda : \sup_{B(y,r) \in \mathcal B_{\lambda}} 
\bigg|\frac{\int_{\mathbb{T}^d} b_{n,y}^{\pm}(x) |\psi_n(x)|^2 \, \dvol(x)}{\int_{\mathbb{T}^d} b_{n,y}^{\pm}(x) \, \, \dvol(x)} -1\bigg| \ge \lambda^{-\delta} \right\}.
\]

 Using Lemma~\ref{lem:localbd} along with the bound $\widehat b_{n,y}^{\pm} \ll r^d$   given by Lemma \ref{lem:BSfns} (iv) (which holds uniformly in $y$),   we get from Chebyshev's inequality as in the proof of Theorem~\ref{thm:small balls} that
\[
\begin{split}
\frac{\# \mathcal S_{\lambda}^{\pm}}{N_{\lambda}}  &
\ll  \frac 1{\lambda^{\frac d2-1-2\delta}} 
\sum_{1 \le |\zeta| \le \lambda^{\theta_2}}
\sum_{\lambda_n=\lambda}|\langle e_{\zeta} \psi_n, \psi_n \rangle|
\sup_{B(y,r) \in \mathcal B_{\lambda}}  \left| \frac{\widehat b_{n,y}^{\pm}(\zeta)}{\widehat b_{n,y}^{\pm}(0)}\right| 
\\
&\ll  \frac 1{\lambda^{\frac d2-1-2\delta}} 
\sum_{1 \le |\zeta| \le \lambda^{\theta_2}}
V_1^{\tmop{loc}}(e_{\zeta},\lambda) \\
&\ll \frac 1{ \lambda^{\frac d2-1-2\delta}}
\sum_{\lambda-\lambda^{2\theta_2}/2\leq t<\lambda} A_d(\lambda,t)\;.
\end{split}
\]
Since we assume $d=3$ and $\lambda\not \equiv 0,4,7\pmod 8$
or $d=4$ and $\lambda$ odd, combining Lemma \ref{lem:Abd}
and \eqref{eq:A4upbd} gives
\[
A_d(\lambda,t) \ll \lambda^{(d-3)/2+o(1)} \gcd(\lambda,t) (\lambda-t)^{(d-3)/2}.
\]
Thus,
\[
\begin{split}
\sum_{\lambda-\lambda^{2\theta_2}/2\le t<\lambda} A_d(\lambda,t)
\ll& \lambda^{(d-3)/2+\theta_2(d-3)+o(1)}  \sum_{\lambda-\lambda^{2\theta_2}/2\le t<\lambda}  \tmop{gcd}(\lambda,t) \\
\ll& \lambda^{(d-3)/2+\theta_2(d-3)+o(1)}  \sum_{e | \lambda}  e \sum_{ \substack{\frac{\lambda-\lambda^{2\theta_2}/2}{e} \le t_0 < \lambda/e} } 1\\
\ll& \lambda^{(d-3)/2+\theta_2(d-1)+o(1)} ,
\end{split}
\]
where in the last step we  bounded the inner sum as $O(\lambda^{2 \theta_2}/e)$ since  if $\lambda^{2\theta_2}/(2e)<1$ then the sum is empty.
Collecting estimates gives
\[
\frac{\# \mathcal S_{\lambda}^{\pm}}{N_{\lambda}}
\ll \lambda^{\theta_2(d-1)-\frac12+3\delta},
\]
which tends to zero if $\theta_1<\theta_2<\frac{1}{2(d-1)}-3\delta$.

Thus, the subset of the ONB $\{\psi_n\}_{\lambda_n=\lambda}$, which consists of eigenfunctions $\psi_n$ with $\lambda_n \notin \mathcal (S_{\lambda}^+ \cup S_{\lambda}^-)$ has
cardinality $N_{\lambda}(1+o(1))$ provided $\theta_1< \theta_2<\frac{1}{2(d-1)}-3\delta$. Repeating the same argument given at the end of
the proof of Theorem \ref{thm:small balls} (see equations \eqref{eq:eqfin1}, \eqref{eq:eqfin2})
we see that each eigenfunction in this subset satisfies
\[
\sup_{B(y,r) \in \mathcal B_{\lambda}} \left| \int_{B(y,r)} |\psi_n(x)|^2 \dvol(x)-\vol(B(y,r)) \right| \ll r^d \lambda^{-\delta}+r^d \lambda^{\theta_1-\theta_2} \; .
\] 
\qed

\begin{remark} \label{rem:caps}
Our argument reduces the problem of small scale quantum ergodicity
to a lattice point estimate, which can be rephrased in terms of statistics of lattice points in caps: 
For each lattice point $\nu\in \vE_\lambda = \{\mu\in \Z^d:|\mu|^2=\lambda\}$, let 
\begin{equation}\label{def of ncap}
\ncap(\nu,Y) = \#(\vE_\lambda\cap \dcap(\nu,Y))-1 = \#\{\mu\in \vE_\lambda:0<|\mu-\nu|\leq Y\}
\end{equation}
be the number of {\em other} lattice points in a cap of size $Y$ about $\nu$. 
In fact we actually show that in
any dimension $d \ge 3$ if
$$\frac 1{N_\lambda} \sum_{|\nu|^2=\lambda} \ncap(\nu,Y) \to 0,\quad {\rm as}\; \lambda\to \infty
$$
then the assertion of Theorem \ref{thm:3dloc} holds
in dimension $d$ at scales $r>Y^{-1+o(1)}$ (we also assume here that $\lambda\not\equiv 0,4,7 \pmod 8$ if $d=3$ and $\lambda$ is odd if $d=4$, for $d \ge 5$ no such restrictions are needed). That is, given the above, small 
scale quantum ergodicity holds in dimension $d$ at scales above $r>Y^{-1+o(1)}$ on every such $\lambda$-eigenspace.
\end{remark}

\section{Massive irregularities}\label{sec:deep}

In this section we are concerned with the existence of a sequence of eigenfunctions $\psi_{\lambda}$ for which the proportion of the $L^2$ mass of $\psi_{\lambda}$
 within small balls
becomes arbitrarily large as $\lambda \rightarrow \infty$. For $d=4$ we show the existence of such a sequence of eigenfunctions $\psi_{\lambda}$
for balls with radii $r_{\lambda} \le \lambda^{-1/6-o(1)}$. On the other hand, for $d=2$ we are able to rule out this behavior
for balls with radii that shrink sufficiently slowly.

\subsection{Blowup for $d=4$} 
Let
 \begin{equation}\label{deeply scarred ef}
 \psi_{\lambda}(x)=\frac{1}{\sqrt{N_\lambda}} \sum_{|\mu|^2=\lambda} e_{\mu}(x)\;.
 \end{equation} 
We show that  at small scales the $L^2$ mass of $\psi_{\lambda}$ blows up in dimension  $d=4$.

\begin{theorem} \label{thm:scar4} 
Let $\psi_{m}=\psi_{\lambda_m}$
be as given in \eqref{deeply scarred ef} in dimension $d=4$.
Then
along the sequence of odd eigenvalues $\lambda_m$ we have
for any sequence of radii  $r_m<\lambda_m^{-1/6-o(1)}$, 
\begin{equation} \notag
 \lim_{m\to \infty} \frac 1{\vol(B(0,r_m))} \int_{B(0,r_m)}|\psi_m(x)|^2 \dvol(x) = \infty.
\end{equation}
\end{theorem}

Note that the result is trivial for $r=o( \lambda^{-1/2})$, because then for $x\in B(0,r)$ we can replace $\psi_\lambda(x) \sim \psi(0)=\sqrt{N_\lambda}$ and then the average of $|\psi_\lambda(x)|^2$ over the ball $B(0,r)$ will be large.  This also implies that for $r \ge \varepsilon \lambda^{-1/2}$ with $\varepsilon>0$ sufficiently small 
\[
\begin{split}
\frac{1}{\tmop{vol}(B(0,r))}
\int_{B(0,r)} |\psi_{\lambda}(x)|^2 \tmop{dvol}(x)
\ge& \frac{1}{\tmop{vol}(B(0,r))}
\int_{B(0,\varepsilon \lambda^{-1/2} )} |\psi_{\lambda}(x)|^2 \tmop{dvol}(x) \\
\gg & \frac{N_{\lambda}}{r^d} \cdot \varepsilon^d \lambda^{-d/2}
\end{split}
\]
in every dimension $d \ge 2$.
Recall for $d \ge 3$, $N_\lambda \gg \lambda^{\frac{d}{2}-1-o(1)}$
provided that $\lambda$ is odd if $d=4$ and if $d=3$, $\lambda \not\equiv 0,4,7 \pmod 8$. For such $\lambda$ the RHS
tends to infinity for $r_{\lambda} \le \lambda^{-\frac{1}{d}-o(1)}$. Theorem \ref{thm:scar4} shows that massive irregularities extend beyond this trivial regime.

For $T\le \sqrt{2\lambda}$ let
\[
S_d(\lambda, T)=\sum_{\lambda-T^2/2\le t <\lambda} A_d(\lambda,t)
\]
and note that in the proof of Lemma \ref{lem:localbd} we saw that 
\begin{equation} \label{eq:Sid}
S_d(\lambda, T)=
\sum_{2 \le |\zeta|^2 \le T^2 } \#\{ \mu : |\mu|^2=\lambda=|\mu+\zeta|^2\}.
\end{equation}
\begin{lemma} \label{lem:deepscar}
Let $\psi_{\lambda}$ be as in \eqref{deeply scarred ef}.
Suppose that $r_\lambda \rightarrow 0$ as $\lambda \rightarrow \infty$. 
Then for any dimension $d \ge 2$
\[
\frac{1}{\tmop{vol} B(0,r_\lambda)}
\int_{B(0, r_{\lambda})} |\psi_{\lambda}(x)|^2 \, \dvol(x) \gg \frac{1+ S_d(\lambda, r_{\lambda}^{-1+o(1)})}{N_\lambda}.
\]
\end{lemma}
 
\begin{remark} 
The RHS is bounded below by the mean value
$$
\frac 1{N_\lambda} \sum_{|\nu|^2=\lambda} \ncap(\nu,r_{\lambda}^{-1+o(1)})  
$$
of the other lattice points in caps of size $r_{\lambda}^{-1+o(1)}$, where $\ncap(\nu,Y)$ is as defined in \eqref{def of ncap}. 
So if this tends to infinity  then the conclusion of Theorem \ref{thm:scar4} holds in dimension $d$
at scales $r_{\lambda} $.
\end{remark}

\begin{proof}
We first construct an auxiliary smooth minorant of $\mathbf 1_{B(0,r_{\lambda})}(x)$ on the torus.
Let $f \in C_0^{\infty}(\mathbb R)$ be a nonzero function such that $0\le f(x) \le 1$  and $ \tmop{supp} f =[-\tfrac12,\tfrac12]$. Let $g:\mathbb R^d \rightarrow \mathbb R$ be given by $g(x)=f(|x|)$ and define $F_{r_{\lambda}}:\T^d\to \R$ by 
$$
F_{r_\lambda}(x)=  \sum_{n \in \mathbb Z^d} (g\ast g)\left(\frac{x+2\pi n}{r_{\lambda}}\right)\;.
$$
Observe that
\[
 (g \ast g)(y)=\int_{\mathbb R^d} f(|x|)f(|y-x|) dx <1.
\]
Also, for $|y|\ge 1$
\[
 0 \le (g \ast g)(y) \le \int_{|y-x|<\frac12, |x|<\frac12} 1 \; dx=0.
\]
It follows that $\mathbf 1_{B(0,1)}(x) \ge (g \ast g)(x) \ge 0$. 
Write
\[
\mathcal F(g\ast g)(\xi)=\int_{\mathbb R^d} (g\ast g)(x) e_{-\xi}(x) \, \frac{dx}{(2\pi)^d}
=\left| \int_{\mathbb R^d} g(x) e_{-\xi}(x) \, \frac{dx}{(2\pi)^d} \right|^2
\]
and note 
by Poisson summation
\begin{equation} \notag
\begin{split}
F_{r_\lambda}(x)=  \sum_{n \in \mathbb Z^d} (g\ast g)\left(\frac{x+2\pi n}{r_{\lambda}}\right)=&r_{\lambda}^d \sum_{\zeta \in \mathbb Z^d} \mathcal F(g\ast g)(r_{\lambda}\zeta) e_{\zeta}(x).
\end{split}
\end{equation}
Hence, $F_{r_\lambda}: \mathbb T^d \rightarrow \mathbb R$ is a smooth minorant of $\mathbf 1_{B(0,r_{\lambda})}(x)$ and has \emph{non-negative} Fourier coefficients. Also, observe
 that $\mathcal F(g \ast g)(\xi)=\mathcal F(g \ast g)(0)+O(|\xi|)$.
From these estimates we get that
\[
\begin{split}
\int_{B(0,r_{\lambda})} |\psi_{\lambda}(x)|^2 \, &\dvol(x) \ge  \int_{\mathbb T^d} F_{r_{\lambda}}(x) |\psi_{\lambda}(x)|^2 \, \dvol(x) \\
=&
\frac{r_{\lambda}^d}{N_\lambda} \sum_{\zeta \in \mathbb Z^d} \mathcal F(g\ast g)(r_{\lambda}\zeta)\#\{ \mu : |\mu|^2=\lambda=|\mu+\zeta|^2\} \\
\ge& \frac{r_{\lambda}^d}{2 N_\lambda} \mathcal F(g\ast g)(0)\left(1+\sum_{0 \neq |\zeta| \le r_{\lambda}^{-1+o(1)}} \#\{ \mu : |\mu|^2=\lambda=|\mu+\zeta|^2\} \right)
\end{split}
\]
by dropping the large frequencies using the non-negativity of $\mathcal F(g\ast g)$ and also noting  
 note that $\mathcal F(g\ast g)(0) = \left|\int_{\R^d} g(x)\frac{dx}{(2\pi)^d} \right|^2>0$.
Applying \eqref{eq:Sid} to the inner sum completes the proof.
\end{proof}

\subsection{Proof of Theorem \ref{thm:scar4}}

By Lemma \ref{lem:deepscar} it suffices to show that  for odd values of  $\lambda \rightarrow \infty$ such that for any sequence  $ r_{\lambda}\ll \lambda^{-\frac 16-o(1)}$, 
we have 
$$S_4(\lambda, r_{\lambda}^{-1+o(1)})/R_4(\lambda) \rightarrow \infty \;.
$$

 By definition, if $T=r^{-1+o(1)}$, 
\begin{equation*}
\begin{split}
S_4(\lambda,T) &= \sum_{0<\lambda-t\leq T^2/2}A_4(\lambda,t) 
\\
&\geq
 \sum_{\substack{0<\lambda-t\leq T^2/2\\t\;{\rm even} }}A_4(\lambda,t) \;.
\end{split}
\end{equation*}
We now assume $r>\lambda^{-1/2}$ so that  $|t| < \lambda$. Applying \eqref{eq:A4id} for odd $\lambda$ we have $A_4(\lambda,t) \geq 8R_3(\lambda^2-t^2)$ so that 
$$
S_4(\lambda,T)\geq \sum_{\substack{0<\lambda-t\leq T^2/2\\t\;{\rm even}}}R_3(\lambda^2-t^2) \;.
$$
Recall that if $n\neq 0,4,7 \pmod 8$ then Siegel's theorem gives $R_3(n)\gg n^{\frac 12-o(1)}$.  Now if $\lambda$ is odd and $t$ is even then $ \lambda^2-t^2 = 1,5\pmod 8$ and in particular Siegel's theorem implies 
\begin{equation*}
R_3(\lambda^2-t^2)\gg (\lambda^2-t^2)^{\frac 12-o(1)}\gg \lambda^{\frac 12-o(1)} (\lambda-t)^{\frac 12}\;.
\end{equation*}
Hence we find 
\begin{equation}\notag
\begin{split}
S_4(\lambda,T)\gg & 
 \lambda^{\frac 12-o(1)}
\sum_{\substack{0<\lambda-t\leq T^2/2\\ t\;{\rm even}}}(\lambda-t)^{\frac 12} \\
= & \lambda^{\frac 12-o(1)}\sum_{\substack{1\leq m\leq T^2/2 \\m\;{\rm odd}}}m^{1/2} 
\gg  \lambda^{\frac 12-o(1)} T^{3}.
\end{split}
\end{equation}
Hence for $T \approx r^{-1+o(1)}$
with  $\lambda^{-1/2}<r\ll \lambda^{-1/6+o(1)}$
$$
S_4(\lambda,r^{-1+o(1)})\gg  \lambda^{\frac 12-o(1)}r^{-3} \;.
$$
Since $R_4(\lambda) \ll \lambda^{1+o(1)} $, we find that  
 along the sequence of odd integers 
$$\frac{S_4(\lambda, r_{\lambda}^{-1+o(1)})}{R_4(\lambda)} \gg \lambda^{-1/2-o(1)}r^{-3}  \rightarrow \infty
$$ 
for $r_{\lambda}\ll \lambda^{-\frac16-o(1)}$.

\qed

\subsection{Ruling out blowup for $d=2$ at certain scales}
The construction of massive irregularities in the previous section used some features particular to high dimensions.  
In fact for $d=2$, we can rule out the existence of this behavior at scales that are not too small, and
expect that massive irregularities do not exist at all scales that are at least slightly above the Planck scale. More precisely, if $d=2$ then for every eigenfunction $\psi_{\lambda}$ we will prove that the
proportion of $L^2$ mass  inside balls with radii $r_{\lambda} > \lambda^{-1/4+o(1)}$ is bounded and we expect this should be true as long as $r_{\lambda} > \lambda^{-1/2+o(1)}$. 

\begin{proposition} 
Let $\psi_{\lambda}(x)$ be an  $L^2(\TT^2, \tmop{dvol})$ normalized
eigenfuction in dimension $d=2$ with eigenvalue $\lambda$. Then for
any ball with radius $r_{\lambda} > \lambda^{-1/4+o(1)}$
\begin{equation} \label{eq:nodscar}
\sup_{y \in \TT^2}\frac{1}{\tmop{vol}(B(y,r))}
\int_{B(y,r)} |\psi_{\lambda}(x)|^2 \dvol(x) \ll 1.
\end{equation}
\end{proposition}

\begin{proof}
Let $b_y^{+}$ be the translated Beurling-Selberg polynomial described in the proof of Theorem \ref{thm:small balls}   which majorizes the indicator function of $B(y,r)$
on $\TT^2$ with length $T=2/r$, so in particular
$\Big|\widehat b_y^+(\zeta)/\tmop{vol}(B(y,r)) \Big| \ll 1$, uniformly for $y \in \TT^2$.
Write
\[
\psi_{\lambda}(x)=\sum_{|\mu|^2=\lambda} c(\mu) e_{\mu}(x)
\]
and argue as in the proof of Lemma \ref{lem:localbd} to get
\[
\begin{split}
\frac{1}{\tmop{vol}(B(y,r))}
\int_{B(y,r)} |\psi_{\lambda}(x)|^2 \dvol(x) \ll& 1+ \sum_{2 \le \ell \le T^2}  \sum_{\substack{\mu,\nu \in \mathbb Z^2 \\ |\mu|^2=\lambda=|\nu|^2 \\  |\mu-\nu|^2=\ell   }} |c(\mu)c(\nu)| \\
\le& 1+   \sum_{|\mu|^2=\lambda} |c(\mu)|^2 \sum_{2 \le \ell \le T^2}\sum_{\substack{|\nu|^2=\lambda \\  |\mu-\nu|^2=\ell   }} 1,
\end{split}
\]
uniformly for $y \in \TT^2$.
To bound the inner sum, let 
\[
M(R,\rho)=\max_{|\mu|=R} \#\{ \nu\in \Z^2 : |\mu|=R=|\nu|, \; |\mu-\nu| \le \rho \}
\]
be the maximal number of lattice  points in an arc of size $\rho$ on the circle of radius $R$. Note that 
\[
\sum_{2 \le \ell \le T^2}\sum_{\substack{|\nu|^2=\lambda \\  |\mu-\nu|^2=\ell   }} 1 \le M(\sqrt{\lambda},T)-1.
\]
Since $T=2/r$ we conclude
\[
\sup_{y \in \TT^2}
\frac{1}{\tmop{vol}(B(y,r))}
\int_{B(y,r)} |\psi_{\lambda}(x)|^2 \dvol(x) \ll M\left(\sqrt{\lambda}, \frac{2}{r}\right).
\] 
A result of Cilleruelo and C\'ordoba \cite{CillerueloCordoba} states that  for any $0<\delta<1/2$, 
$$
M(R,R^{1/2-\delta}) \ll_{\delta} 1\;,
$$ 
thus \eqref{eq:nodscar} holds for $r>\lambda^{-1/4+o(1)}$ as claimed. Moreover, we expect that $M(R,R^{1-\delta}) \ll_\delta 1$; this would imply that \eqref{eq:nodscar} holds for $r_{\lambda}>\lambda^{-1/2+o(1)}$. 
\end{proof}

\end{document}